\documentclass[a4paper,UKenglish,cleveref, autoref, thm-restate, nolineno]{socg-lipics-v2021}

\hideLIPIcs  

\usepackage{xspace}

\newcommand{\R}{\mathbb{R}}
\newcommand{\ER}{\exists\mathbb{R}}
\newcommand{\AS}{allowable sequences\xspace}

\DeclareMathOperator{\conv}{Conv}

\newcommand{\abs}[1]{\left\lvert #1 \right\rvert\xspace}

\renewcommand{\epsilon}{\ensuremath\varepsilon}

\renewcommand{\phi}{\ensuremath{\varphi}}

\newtheorem{longlemma}{Lemma}

\newcommand{\Naturals}{\mathbb{N}\xspace}
\newcommand{\f}[1]{\relax\ifmmode#1\else{$#1$}\fi}

\newcommand{\dimensionA}{\f{d}\xspace}
\newcommand{\dimension}{\f{d}\xspace}

\newcommand{\Grid}{\f{\Gamma}\xspace}

\newcommand{\partitionLength}{\f{n}\xspace}

\title{Splitting Sandwiches Unevenly via Unique Sink Orientations and Rainbow Arrangements  }

\author{Michaela Borzechowski}{Department of Mathematics and Computer Science, Freie Universität Berlin, Germany}{michaela.borzechowski@fu-berlin.de}{}{DFG within GRK~2434 \emph{Facets of Complexity}.}

\author{Sebastian Haslebacher}{Department of Computer Science, ETH Zürich, Switzerland}{sebastian.haslebacher@inf.ethz.ch}{https://orcid.org/0000-0003-3988-3325}{}

\author{Hung P. Hoang}{Algorithms and Complexity Group, Faculty of Informatics, TU Wien, Austria}{phoang@ac.tuwien.ac.at}{https://orcid.org/0000-0001-7883-4134}{Austrian Science Foundation (FWF, projects 10.55776/Y1329 and ESP1136425)}

\author{Patrick Schnider}{Department of Mathematics and Computer Science, University of Basel, Switzerland \\ Department of Computer Science, ETH Zürich, Switzerland}{patrick.schnider@inf.ethz.ch}{https://orcid.org/0000-0002-2172-9285}{}

\author{Simon Weber}{Department of Computer Science, ETH Zürich, Switzerland}{simon.weber@inf.ethz.ch}{https://orcid.org/0000-0003-1901-3621}{}

\authorrunning{M. Borzechowski, S. Haslebacher, H. P. Hoang, P. Schnider, and S. Weber} 

\Copyright{Michaela Borzechowski, Sebastian Haslebacher, Hung P. Hoang, Patrick Schnider, and Simon Weber} 



\category{}

\relatedversion{} 
\acknowledgements{}

\EventEditors{John Q. Open and Joan R. Access}
\EventNoEds{2}
\EventLongTitle{42nd Conference on Very Important Topics (CVIT 2016)}
\EventShortTitle{CVIT 2016}
\EventAcronym{CVIT}
\EventYear{2016}
\EventDate{December 24--27, 2016}
\EventLocation{Little Whinging, United Kingdom}
\EventLogo{}
\SeriesVolume{42}
\ArticleNo{23}

\begin{document}

\maketitle

\begin{abstract}
    The famous Ham-Sandwich theorem states that any $d$ point sets in $\mathbb{R}^d$ can be simultaneously bisected by a single hyperplane. The $\alpha$-Ham-Sandwich theorem gives a sufficient condition for the existence of biased cuts, i.e., hyperplanes that do not cut off half but some prescribed fraction of each point set. We give two new proofs for this theorem. The first proof is completely combinatorial and highlights a strong connection between the $\alpha$-Ham-Sandwich theorem and Unique Sink Orientations of grids. The second proof uses point-hyperplane duality and the Poincar\'{e}-Miranda theorem and allows us to generalize the result to and beyond oriented matroids. For this we introduce a new concept of rainbow arrangements, generalizing colored pseudo-hyperplane arrangements. Along the way, we also show that the realizability problem for rainbow arrangements is $\ER$-complete, which also implies that the realizability problem for grid Unique Sink Orientations is $\ER$-complete.
\end{abstract}

\section{Introduction}
\label{sec:intro}

The famous \emph{Ham-Sandwich theorem}, originally proven by Stone and Tukey in 1942~\cite{StoneTukey}, states that given any $d$ mass partitions or point sets in $\R^d$, there is a hyperplane that simultaneously bisects all of them. As the name suggests, this can be illustrated using the following food-based analogy: assume you have a 3-dimensional sandwich consisting of bread, ham, and cheese, that you want to share with your friend. You want to do this fairly, meaning that both of you get exactly half of the bread, half of the ham, and half of the cheese. The Ham-Sandwich theorem now says that you can always get a fair division with a single straight cut, no matter how you assembled the sandwich\footnote{In his book ``Algorithms in Combinatorial Geometry'', Edelsbrunner illustrates this by saying that such a cut can be found even if the cheese is still in the fridge \cite{EdelsbrunnerBook}.}.

The Ham-Sandwich theorem has initiated the study of \emph{mass partitions}, where the general question is how many mass distributions or point sets can be simultaneously bisected with some specific type of cut. Studied variants include partitions with several hyperplanes \cite{Barba2019,Blagojevic2022,Hubard2020,Hubard2024,Schnider2021}, partitions using general convex sets \cite{Aichholzer:2018gu,Akopyan:2013jt,Blagojevic:2007ij}, partitions using fans and cones \cite{Barany:2002tk,Schnider2019,Soberon2023}, or partitions using fixed shapes \cite{cookie_cutters}. For more information on mass partitions, we refer to the recent survey by Rold\'{a}n-Pensado and Sober\'{o}n \cite{RoldanPensado2022}.

If you like cheese significantly more than your friend does, you might not want to share your sandwich fairly, but you perhaps want to have more of the cheese. A natural question is which such biased partitions are still possible with just a single cut. The \emph{$\alpha$-Ham-Sandwich theorem} gives a sufficient condition for when such cuts are possible. In particular, it states that if your ingredients are \emph{well-separated}, then any biased cut is possible. Formally, well-separation is defined as follows (see also Figure \ref{fig:well-separation} for an illustration).

\begin{definition}
    A family $\mathcal{F}=(A_1,\ldots,A_d)$ of $d$ point sets (or mass distributions) is said to be \emph{well-separated} if for any index set $I \subset [d]$,\footnote{Herein and henceforth, $[d]$ refers to the set $\{1, \dots, d\}$.} there exists a hyperplane that separates the convex hull of (the support of) $\bigcup_{i\in I}A_i$ from the convex hull of (the support of) $\bigcup_{i\in [d]\setminus I}A_i$.
\end{definition}

\begin{figure}[ht]
    \centering
    \includegraphics[width=0.5\linewidth]{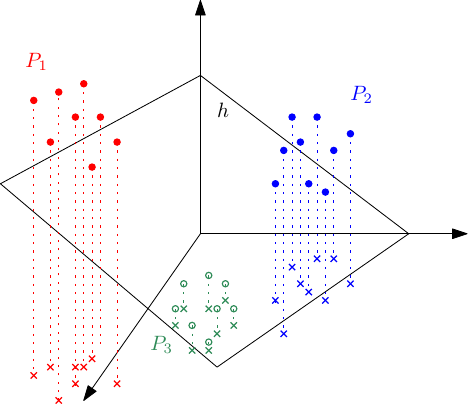}
    \caption{A family of well-separated point sets $\mathcal{F} = (P_1, P_2, P_3)$ in $\R^3$. The hyperplane $h$ separates $P_1\cup P_2$ from $P_3$.
    }
    \label{fig:well-separation}
\end{figure}

The $\alpha$-Ham-Sandwich theorem has been shown both for mass distributions, by B\'{a}r\'{a}ny, Hubard, and Jer\'{o}nimo \cite{baranySlicingConvexSets2008}, as well as for finite point sets, by Steiger and Zhao \cite{steigerGeneralizedHamSandwichCuts2010}. In this manuscript, we focus on the discrete variant for point sets, so we only formally state this version here. However, before getting there, let us remark that we orient a hyperplane $h\in\R^d$ defined by points $p_1\in P_1,\ldots,p_d\in P_d$ according to the orientation of the simplex $(p_1,\ldots,p_d)\in h$, that is, a point $q$ is below $h$ iff the following $(d+1) \times (d+1)$ matrix $A$ has a negative determinant: $A_{j,d+1} = 1$ for all $j \in [d+1]$, $A_{i, \cdot}= p_i$ for all $i \in [d]$ and $A_{d+1,\cdot}=q$. If the matrix $A$ has a positive determinant we say that $q$ is above $h$. If $q$ is on $h$, then $A$ has determinant 0.

With this, we can now formally define $\alpha$-cuts (see also Figure \ref{fig:cuts}) and state the discrete $\alpha$-Ham-Sandwich theorem.

\begin{definition}
    Let $\mathcal{P}=(P_1\ldots,P_d)$ be well-separated and in weak general position\footnote{Weak general position means that any hyperplane containing one point of each of the sets $P_1, \dots, P_d$ cannot contain any other points of $\mathcal{P}$.}.    
    Given $(\alpha_1, \dots, \alpha_d) \in [|P_1|] \times \dots \times [|P_d|]$, an \emph{$(\alpha_1, \dots, \alpha_d)$-cut} is a hyperplane $h_\alpha$ such that for each $i\in[d]$, one point of $P_i$ is on $h_{\alpha}$ and exactly $\alpha_i - 1$ points of $P_i$ are below it.
\end{definition}

\begin{restatable}{theorem}{alphaHS}\label{thm:alphaHS}
Let $\mathcal{P}=(P_1\ldots,P_d)\subset\R^d$ be well-separated and in weak general position.
Then for every $(\alpha_1, \dots, \alpha_d) \in [|P_1|] \times \dots \times [|P_d|]$, there exists a unique $(\alpha_1, \dots, \alpha_d)$-cut.
\end{restatable}

\begin{figure}[ht]
    \centering
    \includegraphics[width=0.5\linewidth]{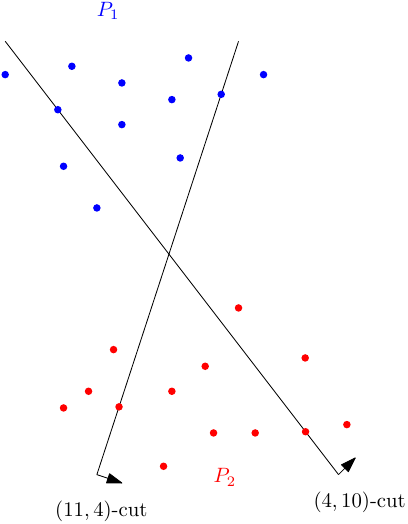}
    \caption{Well-separated point sets in $\R^2$ with two $\alpha$-cuts.}
    \label{fig:cuts}
\end{figure}

The proof of the continuous version for mass distributions by B\'{a}r\'{a}ny, Hubard, and Jer\'{o}nimo \cite{baranySlicingConvexSets2008} uses Brouwer's fixpoint theorem. On the other hand, the proof of the discrete version for point sets by Steiger and Zhao \cite{steigerGeneralizedHamSandwichCuts2010} uses an elegant inductive argument to show that every $\alpha$-cut is unique and then deduces the existence from the pigeonhole principle. However, we think that there is a gap in their inductive argument, which we elaborate on in Appendix~\ref{app:gap}. While it is possible that their proof can be fixed using some additional arguments, we are taking a different route and instead provide two new proofs of the $\alpha$-Ham-Sandwich theorem for point sets.

The first proof, presented in Section \ref{sec:uso_proof}, uses the combinatorial framework of \emph{Unique Sink Orientations}.
A Unique Sink Orientation (USO) is an orientation of the edge set of a hypercube such that every subcube has a unique sink. USOs were formally defined by Szab{\'o} and Welzl \cite{szabo2001usos}, based on earlier work by Stickney and Watson~\cite{stickney1978digraph} who investigated USOs as a combinatorial abstraction of the candidate solutions of a special class of linear complementarity problems.
The concept of USOs has been generalized to \emph{grids}, which are products of complete graphs \cite{gaertner2008grids}. A \emph{grid USO} is an orientation of a grid where every subgrid has a unique sink.

Our second new proof of the $\alpha$-Ham-Sandwich theorem, presented in Section \ref{sec:rainbow}, considers the setting under the well-known point-hyperplane duality: given $d$ sets of hyperplanes with a dual notion of ``well-separated'', we prove the existence of a point that lies above or below the correct number of hyperplanes in each set. In fact, we present a more general result, replacing the arrangement of hyperplanes by a \emph{rainbow arrangement}, a concept that we formally introduce in Section \ref{sec:rainbow_definition}. Informally, a colored generalized arrangement is a family $\mathcal{F}=(H_1, \ldots, H_d)$, where each $H_i$ is a set of pseudo-hyperplanes in $\R^d$. In a rainbow arrangement, we further require that any colorful choice of pseudo-hyperplanes $h_1\in H_1,\ldots,h_d\in H_d$ intersect in a single point. The main result of Section \ref{sec:rainbow_proof} is the following theorem.

\begin{restatable}{theorem}{rainbowHS}\label{thm:rainbowHS}
Let $\mathcal{F}=(H_1, \ldots,H_d)\subset\R^d$ be a well-separated colored generalized arrangement where each color class $H_i$ has size $n_i$. Then for every $(\alpha_1, \dots, \alpha_d) \in [n_1] \times \dots \times [n_d]$, there is a point $x_{\alpha} \in \mathbb{R}^d$ lying on one and above exactly $\alpha_i - 1$ pseudo-hyperplanes of color $i$ for all $i \in [d]$.
If $\mathcal{F}$ is additionally a rainbow arrangement, then the point $x_\alpha$ is unique.
\end{restatable}

Using the topological representation theorem, \Cref{thm:rainbowHS} allows us to deduce a version of the $\alpha$-Ham-Sandwich theorem for oriented matroids (\Cref{cor:matroids}).

In Section \ref{sec:generalized_separation}, we generalize \Cref{thm:rainbowHS} even further by relaxing the well-separation condition to the notion of \emph{$(\beta,\gamma)$-separation}, which we introduce in the same section.

Finally, our notion of rainbow arrangements can be viewed as a higher-dimensional generalization of \emph{bicolored order types}, introduced by Aichholzer and Br\"{o}tzner~\cite{aichholzerBicoloredOrderTypes2024}. In Section~\ref{sec:er_hardness}, we show that deciding whether a rainbow arrangement can be realized by a hyperplane arrangement is $\ER$-complete already in two dimensions (i.e., for bicolored order types, see \cref{theorem:er-hardness}). Given the connection between the $\alpha$-Ham-Sandwich theorem and grid USOs, and based on recent results by Borzechowski, Fearnley, Gordon, Savani, Schnider, and Weber~\cite{borzechowskiTwoChoicesAre2024}, this allows us to conclude that deciding whether a grid USO is realizable is also $\ER$-complete (see \cref{corollary:grid_uso_hard}).

\section{First Proof: Unique Sink Orientations}
\label{sec:uso_proof}

A grid graph is a generalization of the $d$-dimensional hypercube: 
While the latter is the Cartesian product of $d$ copies of $K_2$ (the complete graph on two vertices), the former is the Cartesian product of complete graphs of arbitrary size.
Formally, a $\dimensionA$-dimensional \emph{grid graph}~$\Grid$ parameterized by $\partitionLength_1, \dots, \partitionLength_\dimensionA \in \Naturals_{\geq 2}$ is the graph on $[\partitionLength_1] \times \dots \times [\partitionLength_\dimensionA]$, where two vertices are adjacent if and only if they differ in exactly one coordinate.
We say the grid has $\dimensionA$ \emph{dimensions} and each dimension $i$ has $n_i$ \emph{directions}.

The subgraph $\Grid'$ of $\Grid$ induced by the vertices $V(\Grid') = N_1 \times \dots \times N_\dimension$ for non-empty $N_i \subseteq [\partitionLength_i]$ 
is called an \emph{induced subgrid} of $\Grid$.
Note that if we have $\abs{N_i}=1$ for some $i$, then the induced subgrid loses a dimension.
If $\abs{N_i} \leq 2$ for all $i \in [d]$, we say that $\Grid'$ is a \emph{subcube} of~$\Grid$.
An orientation of the edges of a grid graph is called a \emph{Unique Sink Orientation (USO)} if every induced subgrid has a unique sink (i.e., a unique vertex that has no outgoing edges)~\cite{gaertner2008grids}.

\begin{lemma}\label{lem:GridUSO}
Let $\sigma$ be a grid orientation. Assume that every induced subgrid has a (not necessarily unique) sink, and that $\sigma$ is a USO when restricted to any subcube. Then $\sigma$ is a USO.
\end{lemma}

\begin{proof}
As every subgrid has a sink by assumption, it remains to show that this sink is unique. Assume for the sake of contradiction that there is a subgrid that has two sinks $a=(a_1,\ldots,a_d)$ and $b=(b_1,\ldots,b_d)$. 

Consider now the subcube spanned by $a$ and $b$ (i.e., the subcube induced by $\{a_1, b_1\} \times \dots \times \{a_d, b_d\}$). As $a$ and $b$ were sinks in the subgrid, they are also sinks in this subcube. We have thus found a subcube with two sinks, so $\sigma$ restricted to this subcube is not a USO. This is a contradiction to the second assumption.
\end{proof}

Let $\mathcal{P} = (P_1, \dots, P_d) \subset \R^d$ be a well-separated point set in weak general position.
For $i \in [d]$, let $P_i = \{p^i_1, \dots, p^i_{|P_i|} \}$.

Let $\Grid_{\mathcal{P}}$ be the grid graph parameterized by $|P_1|, \dots, |P_d|$.
We recall the following orientation $\sigma_{\mathcal{P}}$ on $\Grid_{\mathcal{P}}$, which was presented in~\cite{borzechowskiTwoChoicesAre2024} and is illustrated in Figure \ref{fig:uso}.
For any edge between two vertices $v = (a_1, \dots, a_{i-1}, a_i, a_{i+1}, \dots, a_d)$ and $v' = (a_1, \dots, a_{i-1}, a'_i, a_{i+1}, \dots, a_d)$ of $\Grid_{\mathcal{P}}$, let $h$ be the colorful hyperplane spanned by the points $p^1_{a_1},\ldots,p^{i - 1}_{a_{i-1}},p^i_{a_i},p^{i + 1}_{a_{i+1}},\ldots,p^d_d$, and let $h'$ be the colorful hyperplane spanned by $p^1_{a_1},\ldots,p^{i - 1}_{a_{i-1}},p^i_{a'_i},p^{i + 1}_{a_{i+1}},\ldots,p^d_d$.
If $p^i_{a'_i}$ lies above $h$, we orient $v'$ towards $v$; otherwise orient $v$ towards $v'$. 
It turns out that $p^i_{a'_i}$ lies above $h$ if and only if $p^i_{a_i}$ lies below the hyperplane $h'$,\footnote{This can be seen by projecting to $\R^2$ such that $h \cap h'$ is mapped to the origin.} so this orientation is well-defined.

\begin{figure}[ht]
    \centering
    \includegraphics[width=0.9\linewidth]{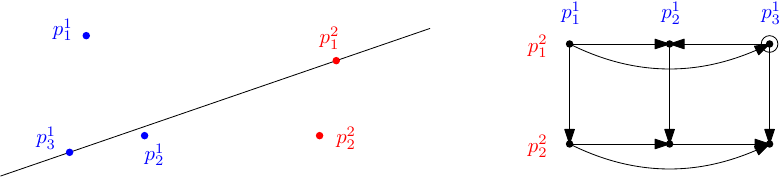}
    \caption{Well-separated point sets in $\R^2$ with the corresponding grid orientation. The line drawn on the left is represented by the highlighted vertex on the right.}
    \label{fig:uso}
\end{figure}

The following lemma is key to our first proof of \cref{thm:alphaHS}.

\begin{lemma}\label{lem:OrientationIsUSO}
Let $\mathcal{P}=(P_1,\ldots,P_d)\subset \mathbb{R}^d$ be well-separated and in weak general position. Then the orientation $\sigma_{\mathcal{P}}$ is a grid USO.
\end{lemma}

Note that this lemma was already proved in~\cite{borzechowskiTwoChoicesAre2024}.
However, their proof uses \cref{thm:alphaHS}.
To avoid the cyclic dependency, we provide here an alternative proof that uses the following version of the $\alpha$-Ham-Sandwich theorem for mass distributions.

\begin{theorem}[B\'{a}r\'{a}ny, Hubard, and Jer\'{o}nimo \cite{baranySlicingConvexSets2008}]
\label{thm:alpha_mass}
    Let $K_1, \dots, K_d$ be well-separated convex bodies in $\R^d$, and $\alpha_1, \dots, \alpha_d$ given constants with $0 \leq \alpha_i \leq 1$ for $i \in [d]$.
    Then there is a unique hyperplane $h$ such that for all $i \in [d]$, $h$ intersects $K_i$, and the proportion of the volume of $K_i$ below $h$ is $\alpha_i$.
\end{theorem}

\begin{proof}[Proof of \cref{lem:OrientationIsUSO}]
We first show that every subgrid has a sink. 
By the definition of $\sigma_{\mathcal{P}}$, a subgrid is spanned by subsets $P'_i\subseteq P_i$ and a sink corresponds to a $(1,\dots,1)$-cut. 

As any subset of a well-separated point set in weak general position is again well-separated and in weak general position, the existence of such a $(1,\dots,1)$-cut is guaranteed by applying \cref{thm:alpha_mass} with the convex sets being the convex hulls of $P'_1, \dots, P'_d$, and $\alpha_i = 0$ for $i \in [d]$.

We now argue that every subcube is a USO. 
A subcube is spanned by at most two points per $P_i$. 
Let $Q_i$ be the set of these at most two points.
As argued before, $(Q_1, \dots, Q_d)$ is well-separated and in weak general position.
Further, the convex hull of the at most two points per $P_i$ is a line segment.
Hence, by applying \cref{thm:alpha_mass} on these convex hulls, we obtain the existence of all possible $(\alpha_1, \dots, \alpha_d)$-cuts of $(Q_1, \dots, Q_d)$ for all $(\alpha_1, \dots, \alpha_d) \in [|Q_1|] \times \dots \times [|Q_d|]$.
Hence, in the subcube, each possible \emph{outmap} (i.e., a binary vector at each vertex that encodes whether the incident edge along each dimension is incoming or outgoing) is present.
This implies that the subcube is indeed a USO~\cite{szabo2001usos}.

The claim now follows from Lemma \ref{lem:GridUSO}.
\end{proof}

The $\alpha$-Ham-Sandwich theorem for point sets (\cref{thm:alphaHS}) now follows easily.

\alphaHS*

\begin{proof}
By Lemma \ref{lem:OrientationIsUSO}, the orientation $\sigma_{\mathcal{P}}$ is a grid USO. 
We define a function $r : [|P_1|] \times \dots \times [|P_d|] \to \{0, \dots, |P_1|-1\} \times \dots \times \{0, \dots, |P_d|-1\}$ that assigns to each vertex of the grid graph $\Grid_{\mathcal{P}}$ a tuple $(a_1, \dots, a_d)$ such that $a_i$ is the number of the vertex's outgoing edges in the dimension $i$.
Then \cite[Theorem 2.14]{gaertner2008grids} states that $r$ is a bijection.
This implies that for every $(\alpha_1,\ldots,\alpha_d) \in [P_1] \times \dots \times [P_d]$, there is a unique $(\alpha_1,\ldots,\alpha_d)$-cut, as required.
\end{proof}

\section{Second Proof: Rainbow Arrangements and the Poincar\'{e}-Miranda Theorem}
\label{sec:rainbow}

We provide definitions and preliminaries on the Poincar\'{e}-Miranda Theorem in Sections~\ref{sec:rainbow_definition} and~\ref{sec:poincare-miranda} and present the proof of \Cref{thm:rainbowHS} in Section~\ref{sec:rainbow_proof}.

\subsection{Point-Hyperplane Duality and Rainbow Arrangements}
\label{sec:rainbow_definition}

In this subsection, we argue why our $\alpha$-Ham-Sandwich theorem for rainbow arrangements (\cref{thm:rainbowHS}) is a generalization of the one for point sets (\cref{thm:alphaHS}).
We first consider the point-hyperplane dual of \cref{thm:alphaHS}.
The dual of every input point is a hyperplane, and the well-separation of the point set translates to the property that for every subset $S$ of $[d]$, there exists a point above all the dual hyperplanes of the points in $\cup_{i \in S} P_i$ and below those of the points in $\cup_{i \notin S} P_i$.
The dual of an $(\alpha_1, \dots, \alpha_d)$-cut is then a point lying on one dual hyperplane and above exactly $\alpha_i-1$ hyperplanes of each color $i$.
\cref{thm:alphaHS} is equivalent to the statement that such a point is unique for every $(\alpha_1, \dots, \alpha_d)$.

With this interpretation in mind, we now define our generalized arrangements, where we define an \emph{oriented pseudo-hyperplane} as a subset of $\R^d$ that is homeomorphic to $\R^{d - 1}$ and divides $\R^d$ into two parts: a positive and a negative side. We say a point $p$ lies \emph{above} (\emph{below}) an oriented pseudo-hyperplane $H$ if and only if $p$ is contained in the positive (negative) side.

\begin{definition}[Generalized Arrangement, Rainbow Arrangement]
\label{def:generalized_arrangments}
A family $\mathcal{F}$ of oriented pseudo-hyperplanes in $\mathbb{R}^d$ is called a \emph{generalized arrangement} if for all $k \in [d + 1]$, the intersection of any $k$ of them is a (not necessarily connected) $(d-k)$-dimensional manifold\footnote{For the sake of convenience, we consider the empty set to be a manifold of any dimension. In particular, it is the only manifold of dimension $-1$.} and $\mathbb{R}^d\setminus\bigcup\mathcal{F}$ has finitely many connected components (which we call \emph{cells}).

A generalized arrangement is \emph{colored} if $\mathcal{F}$ is partitioned into $d$ pairwise disjoint subfamilies $H_1,\ldots,H_d$. We think of this as coloring each oriented pseudo-hyperplane with one of $d$ colors and call the $H_i$'s \emph{color classes}.

A colored generalized arrangement is a \emph{rainbow arrangement} if for each choice $h_1\in \nolinebreak H_1,\ldots,$ $h_d\in H_d$ of one oriented pseudo-hyperplane per color class, we have that the intersection $h_1\cap\ldots\cap h_d$ is a single point.
\end{definition}

Note that our arrangements differ significantly from the well-known concept of \emph{pseudo-hyperplane arrangements}, where the intersection of any $k$ pseudo-hyperplanes is required to be homeomorphic to $\R^{d - k}$. In particular, while every pseudo-hyperplane arrangement is also a generalized arrangement, the converse is not true. 

The following definition of well-separation is a dual version of the notion of well-separation for point sets.

\begin{definition}[Well-Separated]
A colored generalized arrangement $\mathcal{F}=(H_1, \dots, H_d)$ in $\mathbb{R}^d$ is \emph{well-separated} if for every sign vector $s \in \{+, -\}^d$, there exists a point $x_s \in \mathbb{R}^d$ such that 
\begin{itemize}
    \item $x_s$ lies in an unbounded cell,
    \item $x_s$ lies above all oriented pseudo-hyperplanes in $H_i$ iff $s_i = +$,
    \item and $x_s$ lies below all oriented pseudo-hyperplanes in $H_i$ iff $s_i = -$,
\end{itemize}
for all $i \in [d]$.
\end{definition}

\subsection{The Poincar\'{e}-Miranda Theorem}
\label{sec:poincare-miranda}

In our proof, we will use the Poincar\'{e}-Miranda theorem, which we briefly recall.

\begin{theorem}[Poincar\'{e}-Miranda theorem]
Consider $d$ continuous functions
\[f_1,\ldots,f_d: [-1,1]^d\rightarrow \mathbb{R}.\]
Assume that for each variable $x_i$, the function $f_i$ is non-positive whenever $x_i=-1$ and non-negative whenever $x_i=1$.
Then there exists a point $x^*$ in $[-1,1]^d$ with $f_i(x^*)=0$ for all $i\in\{1,\ldots,d\}$.
\end{theorem}

The theorem was first stated by Poincar\'{e}~\cite{poincare1883} without proof. It was later shown by Miranda~\cite{miranda1940} that it is equivalent to Brouwer’s fixed point theorem.
Since then, the theorem has found numerous applications and generalizations; see, e.g., \cite{ariza2019poincare} and the references therein.

\subsection{The Second Proof}
\label{sec:rainbow_proof}

We define the \emph{$k$-level} of a family~$H$ of oriented pseudo-hyperplanes to be the set of points that lie (i) on at least one oriented pseudo-hyperplane of $H$ and (ii) on or above exactly $k$ oriented pseudo-hyperplanes of $H$. 
We are now ready to prove our generalization of the $\alpha$-Ham-Sandwich theorem.

\rainbowHS*

We prove the existence of a point $x_{\alpha}$ by showing that the corresponding $\alpha_i$-levels intersect, see Figure \ref{fig:rainbow_cut} for an illustration. The uniqueness for rainbow arrangement additionally uses a counting argument.

\begin{figure}[ht]
    \centering
    \includegraphics[width=0.7\linewidth]{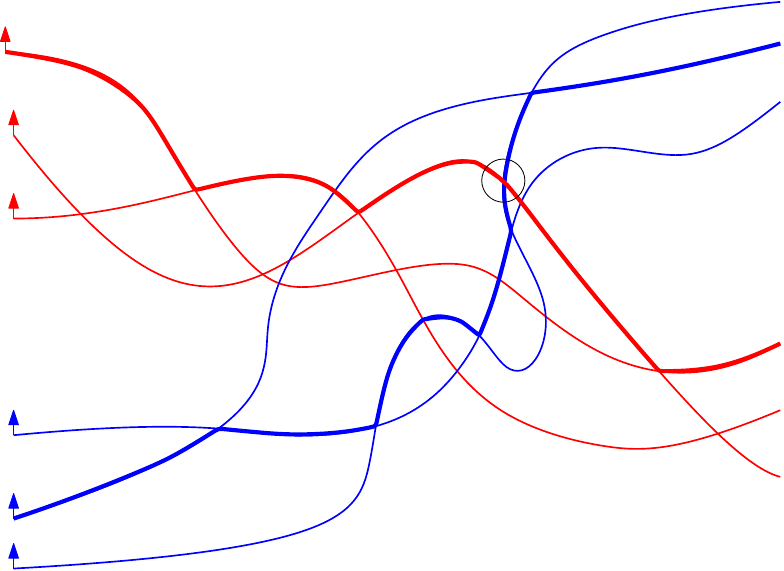}
    \caption{A rainbow arrangement with colors red and blue. The red 3-level and the blue 2-level are highlighted and intersect in a unique point.}
    \label{fig:rainbow_cut}
\end{figure}

\begin{proof}
By assumption $\mathcal{F}$ is well-separated, so for each vector $s \in \{+, -\}^d$ there is a point $x_s$ in an unbounded cell that lies above or below all oriented pseudo-hyperplanes of a color class. Consider a subset $I\subset [d]$ and a family of vectors $s^1,\ldots,s^k$ in $\{+, -\}^d$ for which $s_i^1=\ldots =s_i^k$ for all $i\in I$. Let $\mathcal{F_{I}}=\cup_{i\in I} H_i$ be the generalized arrangement consisting only of oriented pseudo-hyperplanes of color classes in $I$. We claim that the points $x_{s^1},\cdots,x_{s^k}$ lie in a common unbounded cell of $\mathcal{F_{I}}$. Indeed, assume for the sake of contradiction that two of the points $x_{s^i}$ and $x_{s^j}$ lie in different unbounded cells. Then there is an oriented pseudo-hyperplane of $\mathcal{F_{I}}$ separating them. However, by definition, $x_{s^i}$ and $x_{s^j}$ must lie on the same side of any oriented pseudo-hyperplane in~$\mathcal{F_{I}}$, a contradiction.

Now consider the cube $[-1,1]^d$.
Note that each face of the cube can be encoded by a vector $z$ in $\{*,-1,1\}^d$, such that the face contains all vertices $v$ of the cube such that $v_i = z_i$ for all $i$ with $z_i \neq *$.
We call such a face \emph{characterized} by the set $\{i \mid z_i \neq *\}$. 

Next, consider the following embedding of the cube into $\R^d$: Each vertex $v$ in $\{-1,1\}^d$ is mapped to the point $x_s$ with $s \in \{+,-\}^d$ such that $v_i = -1$ if and only if $s_i = -$ for $i \in [d]$.
Then as we already argued above, for every face of the cube characterized by a set $I \subseteq [d]$, all its vertices are in a common unbounded cell of $\mathcal{F}_I$, so we can extend the embedding defined on the facets of a face to the entire face.

After applying a homeomorphism on $\mathbb{R}^d$ that maps this embedding to the cube $[-1,1]^d\subset\mathbb{R}^d$ we can thus assume that all oriented pseudo-hyperplanes of the color class $H_i$ have the facet $x_i=-1$ on their negative side, and the facet $x_i=1$ on their positive side.

Consider now the $\alpha_i$-level $L$ of some color class $H_i$. We define a function $f_i$ for which $f_i(x)=0$ if $x$ is on $L$, $f_i(x)>0$ if $x$ is above $L$ and $f_i(x)<0$ if $x$ is below $L$. To achieve this, set $f_i(x) = 0$ for all points $x$ that lie on $L$. For all other points $x$, let $d(x, L)$ denote the distance of $x$ to $L$. If $x$ is below $L$, we set $f_i(x) = -d(x, L)$, and if $x$ is above $L$, we set $f_i(x) = d(x, L)$. Since the distance is a continuous function, $f_i$ is continuous. 

By construction, the functions $f_i$ satisfy the conditions of the Poincar\'{e}-Miranda theorem.
We thus get that there is a point $x^*$ in $[-1,1]^d$ for which $f_i(x^*)=0$ for all $i\in\{1,\ldots,d\}$.
Recall that by the definition of a generalized arrangement, no point can lie on $d + 1$ of our pseudo-hyperplanes. Thus, $x^*$ is a point lying on exactly one and above exactly $\alpha_i - 1$ oriented pseudo-hyperplanes of color $i$ for all $i \in [d]$, as required.

Note that $x^*$ lies on one oriented pseudo-hyperplane of each color class. If the arrangement is a rainbow arrangement, then there are only $n_1\cdot\ldots\cdot n_d$ many possible locations for $x^*$. As each of these points can only be the solution for one $\alpha$-vector and there are $n_1\cdot\ldots\cdot n_d$ different $\alpha$-vectors, it follows that for each $\alpha$-vector the solution must be unique.
\end{proof}

By the famous topological representation theorem of Folkman and Lawrence~\cite{folkmanOrientedMatroids1978}, \Cref{thm:rainbowHS} also has implications for oriented matroids. To explain this, we will assume familiarity with oriented matroids as presented for example in the Handbook of Discrete and Computational Geometry~\cite{richter-gebertOrientedMatroids1997}. In particular, we consider uniform oriented matroids of rank $d + 1$. We denote such an oriented matroid by $M = (E, \mathcal{L})$, where $E$ is the ground set and $\mathcal{L} \subseteq \{-, 0, +\}^E$. Coloring $E$ corresponds to a partition into classes (colors) as $E = E_1 \sqcup \dots \sqcup E_d$. If we use $c(e)$ to denote the color of element $e \in E$, then we take well-separation to mean that for all $s \in \{- , + \}^d$, there exists a covector $v \in \mathcal{L}$ with $v_e = s_{c(e)}$ for all $e \in E$. By applying the Folkman-Lawrence representation theorem and then \Cref{thm:rainbowHS}, we thus get the following corollary in this setting. 

\begin{corollary}[$\alpha$-Ham-Sandwich for Oriented Matroids]\label{cor:matroids}
    Let $M = (E, \mathcal{L})$ be a uniform oriented matroid of rank $d + 1$ in covector representation with $\mathcal{L} \subseteq \{-, 0, +\}^E$. Assume that $M$ is well-separated, and in particular, $E$ is partitioned into $d$ classes $E = E_1 \sqcup \dots \sqcup E_d$ (colors). Then for every $\alpha \in [|E_1|] \times \dots \times [|E_d|]$, there is a covector $v^{\alpha} \in \mathcal{L}$ with exactly one zero and exactly $(\alpha_i - 1)$ minuses on the subset $E_i$, for all $i \in [d]$. 
\end{corollary}

\section{Generalizing Well-Separation}\label{sec:generalized_separation}
Using the Poincar\'{e}-Miranda theorem, we can prove an even more general statement by relaxing the definition of well-separation. (Recall the definition of a $k$-level of a set of oriented pseudo-hyperplanes, as defined in \cref{sec:rainbow_proof}.) 

\begin{definition}[$(\beta,\gamma)$-separated]
\label{def:generalized_separation}
Let $\mathcal{F}=(H_1, \dots, H_d)$ be a colored generalized arrangement in $\mathbb{R}^d$, where each color class $H_i$ has size $n_i$. Let $\beta=(\beta_1,\ldots,\beta_d)$ and $\gamma=(\gamma_1,\ldots,\gamma_d)$ be vectors such that $\beta_i, \gamma_i \in [n_i]$ and $\beta_i\leq \gamma_i$ for all $i \in [d]$.
The colored generalized arrangement $\mathcal{F}$ is called \emph{$(\beta,\gamma)$-separated} iff for every sign vector $s \in \{+, -\}^d$, there exists a point $x_s \in \mathbb{R}^d$ such that 
\begin{itemize}
	\item $x_s$ lies in an unbounded cell,
    \item $x_s$ lies above the $\gamma_i$-level of $H_i$ iff $s_i = +$,
    \item and $x_s$ lies below the $\beta_i$-level of $H_i$ iff $s_i = -$,
\end{itemize}
for all $i \in [d]$.
\end{definition}

Note that by setting $\beta=(1,\ldots, 1)$ and $\gamma=(n_1,\ldots,n_d)$ we recover the definition of well-separated.

\begin{theorem}
\label{thm:beta-gamma}
Let $\mathcal{F}=(H_1, \ldots,H_d)$ be a $(\beta,\gamma)$-separated colored generalized arrangement where each color class $H_i$ has size $n_i$. Then for every $(\alpha_1, \dots, \alpha_d) \in \{\beta_1,\ldots,\gamma_1\} \times \dots \times \{\beta_d,\ldots,\gamma_d\}$, there is a point $x_{\alpha} \in \mathbb{R}^d$ lying on one and above exactly $\alpha_i - 1$ oriented pseudo-hyperplanes of color $i$ for all $i \in [d]$.
\end{theorem}

\begin{proof}
By assumption $\mathcal{F}$ is $(\beta,\gamma)$-separated, so for each vector $s \in \{+, -\}^d$ there is a point $x_s$ in an unbounded cell that lies above or below all relevant levels of a color class as stated in \cref{def:generalized_separation}. As the arrangement has finitely many cells, we can place a large enough ball $B$ that contains all the bounded cells. Observe that we can assume $x_s$ to lie outside of~$B$. Consider a subset $I\subseteq [d]$ and a family of vectors $s^1,\ldots,s^k$ for which $s_i^1=\ldots =s_i^k$ for all $i\in I$. Let $\mathcal{F_{I}}$ be the union of all relevant levels, that is, the $\alpha_i$-levels for $\alpha_i \in \{\beta_i, \ldots, \gamma_i\}$, of the color classes $i \in I$. We claim that the points $x_{s^1},\ldots,x_{s^k}$ lie in a common unbounded cell of $\mathcal{F_{I}}$. Indeed, assume for the sake of contradiction that two of the points $x_{s^i}$ and $x_{s^j}$ lie in different unbounded cells. Then there is a level of $\mathcal{F_{I}}$ separating them outside of $B$. However, by definition, $x_{s^i}$ and $x_{s^j}$ must lie on the same side of any relevant level in $\mathcal{F_{I}}$.

We thus indeed have that all the points $x_{s^1},\ldots,x_{s^k}$ lie in a common unbounded cell of $\mathcal{F_{I}}$ outside of $B$. 
Now we use the same argument as in the proof of \cref{thm:rainbowHS}.
In particular, we can again find an embedding of the cube $[-1,1]^d$ where each face $F_I$ characterized by a subset $I \subseteq [d]$ 
lies in this common unbounded cell of $\mathcal{F_{I}}$. After applying a homeomorphism on $\mathbb{R}^d$ that maps this embedding to the cube $[-1,1]^d\subset\mathbb{R}^d$ we can thus assume that all relevant levels of the color class $\mathcal{F}_i$ have the facet $x_i=-1$ on their negative side, and the facet $x_i=1$ on their positive side.

We now take the same functions $f_i$ as in the proof of Theorem \ref{thm:rainbowHS}, which satisfy the conditions of the Poincar\'{e}-Miranda theorem.
We thus get that there is a point $x^*$ in $[-1,1]^d$ for which $f_i(x^*)=0$ for all $i\in\{1,\ldots,d\}$.
\end{proof}

We think that it is instructive to translate \Cref{thm:beta-gamma} back to the original setting of colored point sets in $\R^d$ (i.e., the primal setting). Note that in \Cref{def:generalized_separation}, the condition that $x_s$ must lie in an unbounded face has technical reasons: In the proof of \Cref{thm:beta-gamma}, it allows us to easily embed the cube that we need for the Poincar\'{e}-Miranda theorem. However, we can get away without this assumption in the case of (straight) hyperplanes: In that case, the $k$-levels are guaranteed to be connected and piecewise linear, and thus we can embed the cube without assuming that the points $x_s$ lie in unbounded faces. Therefore, when going to the primal setting, we can omit this technical assumption and define $(\beta, \gamma)$-separation for point sets as follows.

\begin{definition}
    Let $\mathcal{P} = (P_1, \dots, P_d ) \subset \R^d$ be in weak general position and let $\beta, \gamma \in [|P_1|] \times \dots \times [|P_d|]$ with $\beta_i \leq \gamma_i$ for all $i \in [d]$ be arbitrary. We say that $\mathcal{P}$ is $(\beta, \gamma)$-separated if for every sign vector $s \in \{+, -\}^d$, there exists a hyperplane $h_s$ such that 
    \begin{itemize}
        \item $h_s$ lies strictly above at least $\gamma_i$ points of $P_i$ iff $s_i = +$, 
        \item and $h_s$ lies above at most $\beta_i - 1$ points of $P_i$ iff $s_i = -$, 
    \end{itemize}
    for all $i \in [d]$.
\end{definition}

By dualizing the point set and applying \Cref{thm:beta-gamma} for the dual hyperplanes, we thus get the following corollary for $(\beta, \gamma)$-separated point sets.

\begin{corollary}
    Let $\mathcal{P} = (P_1, \dots, P_d ) \subset \R^d$ be $(\beta, \gamma)$-separated for some $\beta, \gamma \in [|P_1|] \times \dots \times [|P_d|]$ and in general position\footnote{We require general position (no $d + 1$ points on a common hyperplane) to ensure that the dual arrangement is indeed a generalized arrangment. However, one could easily adapt our proofs to also make this work for weak general position.}. Then for every $(\alpha_1, \dots, \alpha_d) \in \{\beta_1,\ldots,\gamma_1\} \times \dots \times \{\beta_d,\ldots,\gamma_d\}$, there exists a unique $(\alpha_1, \dots, \alpha_d)$-cut.
\end{corollary}

\section{Rainbow Arrangements and Bicolored Stretchability}
\label{sec:er_hardness}

In this section, we take a closer look at well-separated rainbow arrangements in $\R^2$, where pseudo-hyperplanes are also called pseudo-lines. Concretely, such an arrangement consists of well-separated blue and red oriented pseudo-lines such that every blue and red pseudo-line intersect (and thus cross) exactly once.

\begin{figure}[ht]
\label{fig:bicolored_unstretchable}
    \centering
    \includegraphics[width=0.5\linewidth]{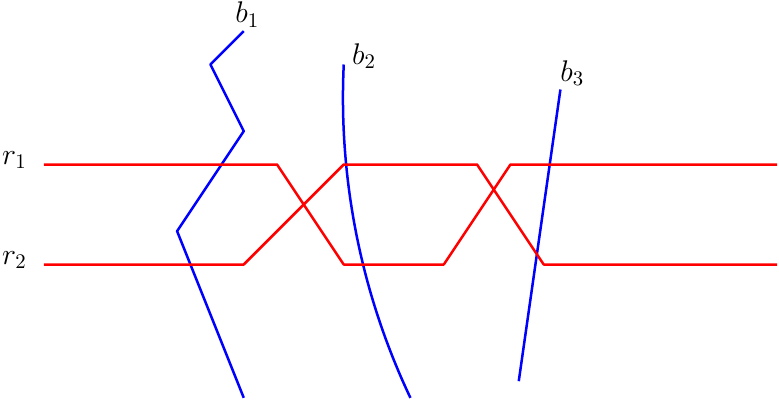}
    \caption{A well-separated rainbow arrangement in $\R^2$. Red pseudo-lines are oriented from left to right, blue pseudo-lines from top to bottom. 
    Note that in order to intersect the blue pseudo-lines in this order, the two red pseudo-lines need to cross twice. In particular, this means that there is no combinatorially equivalent arrangement using straight lines, i.e., this is a NO-instance of bicolored stretchability.}
    \label{fig:unstretchable}
\end{figure}

An interesting question is whether a given rainbow arrangement can be realized using (straight) lines. In the following, we will prove that deciding this is $\ER$-complete. Concretely, we prove that the following formalization of what we call \emph{bicolored stretchability}, is $\ER$-complete.

\begin{definition}[Bicolored Stretchability]
    Given a combinatorial description of a well-separated rainbow arrangement in $\R^2$ (by specifiying for each blue pseudo-line, the order in which it crosses the red pseudo-lines, and vice versa), \emph{bicolored stretchability} is the problem of deciding whether there exists a two-colored well-separated line arrangement with the same combinatorial description.
\end{definition}

In particular, in a YES-instance of bicolored stretchability, all colorful crossings along every line have to happen in the prescribed order. 

Observe that, given an instance of bicolored stretchability and a guess for the line arrangement (specifying each line and its color), we can efficiently verify (in a real-RAM machine) whether this guess indeed describes a well-separated line arrangement that is combinatorially equivalent to the bicolored stretchability instance (well-separation can be checked e.g.\ by comparing the slopes of the lines). Erickson, van der Hoog, and Miltzow~\cite{ericksonSmoothingGapNP2024} conveniently proved that such a verification algorithm on a real-RAM machine is enough to prove membership in $\ER$. 

\begin{figure}[ht]
    \centering
    \includegraphics[width=0.5\linewidth]{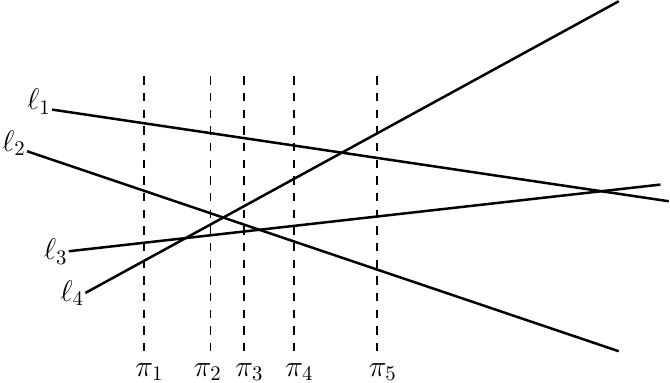}
    \caption{A YES-instance of \AS. The sequence of permutations $\pi_1, \dots, \pi_5$ can be found in the arrangement by sweeping from left to right.}
    \label{fig:allowable}
\end{figure}

In order to prove $\ER$-hardness, we will provide a reduction from the $\ER$-complete problem \emph{\AS}~\cite{hoffmannUniversalityTheoremAllowable2018}.
This problem can be formulated both in its primal form (using point sets) as well as its dual form (using line arrangements). For more details regarding this duality, see e.g.,\@~\cite{goodmanAllowableSequencesOrder1993}. The following dual formulation will be convenient for our purpose.

\begin{definition}[Allowable Sequences]
    Given a sequence of permutations $\pi_1, \dots, \pi_k$
    of $[n]$, where $\pi_{i + 1}$ differs from $\pi_i$ by a swap of two adjacent items for all $i \in [k - 1]$, \emph{allowable sequences} is the problem of deciding whether there exists a line arrangement of $n$ lines such that sweeping the arrangement from left to right and recording how the relative order of the lines changes yields a sequence of permutations that contains the sequence $\pi_1, \dots, \pi_k$.
\end{definition}

Consider now such a sequence $\pi_1, \dots, \pi_k$ of permutations of $[n]$, i.e.,\ an instance of \AS. We build an instance of bicolored stretchability as follows: We introduce $n + 4$ red pseudo-lines $r_{-1}, r_0, r_1, \dots, r_n, r_{n + 1}, r_{n + 2}$ and $2k$ blue pseudo-lines $b_1, \dots, b_k$ and $b'_1, \dots, b'_k$. We call the four red lines $r_{-1}, r_0, r_{n + 1}, r_{n + 2}$ control lines. For each $i \in [k]$, we require both $b_i$ and $b'_i$ to first intersect $r_{-1}, r_0$ in order, then $r_1, \dots, r_n$ in the order of the permutation $\pi_i$, and finally $r_{n + 1}$ and $r_{n + 2}$ in order. Conversely, for each $i \in [n]$, we require $r_i$ to intersect the $2k$ blue pseudo-lines in the order $b_1, b'_1, b_2, b'_2, \dots, b_k, b'_k$. Finally, the control lines are specified such that
\begin{itemize}
    \item $r_{-1}$ intersects the blue pseudo-lines in the order $b'_1, b_1, \dots, b'_k, b_k$,
    \item $r_0$ intersects the blue pseudo-lines in the order $b_1, b'_1, \dots, b_k, b'_k$,
    \item $r_{n + 1}$ has to first intersect $b_1, \dots, b_k$ and then $b'_1, \dots, b'_k$ in that order, 
    \item and $r_{n + 2}$ has to first intersect $b_k, \dots, b_1$ and then $b'_k, \dots b'_1$.
\end{itemize}

Clearly, this construction can be implemented in polynomial time and it remains to prove correctness. However, before getting to the correctness proof, we first observe that this indeed yields a combinatorial description of a rainbow arrangement.

\begin{figure}[ht]
    \centering
    \includegraphics[width=0.5\linewidth]{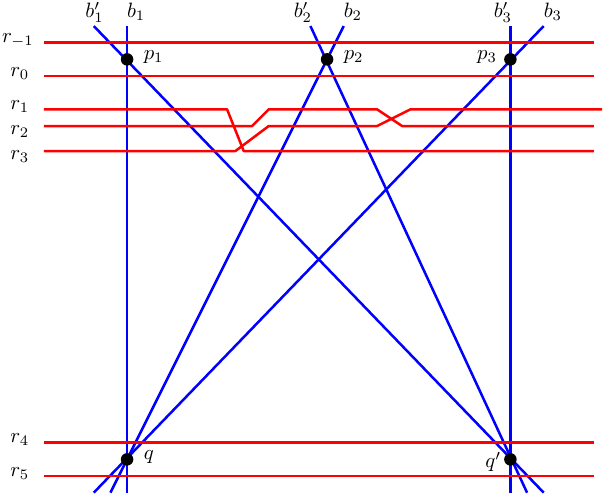}
    \caption{An illustration of the proof of \Cref{lemma:bic_stretch_1} with $k = n = 3$. Concretely, $r_{-1}, r_0, r_4, r_5$ are the four control lines.}
    \label{fig:bic_stretch_1}
\end{figure}

\begin{lemma}
\label{lemma:bic_stretch_1}
    There exists a well-separated rainbow arrangement with the combinatorial description described in the construction above. 
\end{lemma}
\begin{proof}
    This proof is illustrated in \Cref{fig:bic_stretch_1}. We start by drawing the four control lines as straight horizontal lines in the order (top to bottom) $r_{-1}, r_0, r_{n + 1}, r_{n + 2}$, making sure that the pairs $r_{-1}, r_0$ and $r_{n+ 1}, r_{n + 2}$ are relatively close to each other while there is a big gap between $r_0$ and $r_{n + 1}$. Next, we place $k$ points $p_1, \dots, p_k$ with increasing $x$-coordinates between $r_{-1}$ and $r_0$, and we place two points $q$ and $q'$ with increasing $x$-coordinates between $r_{n + 1}$ and $r_{n + 2}$. The blue lines $b_i$ and $b'_i$ are then drawn as straight lines where $b_i$ goes through $p_i$ and $q$, and $b'_i$ goes through $p_i$ and $q'$ for all $i \in [k]$. It remains to draw $r_1, \dots, r_n$ as pseudo-lines just below $r_0$. Note that just below $r_0$, the blue lines read as $b'_1, b_1, \dots, b'_k, b_k$ from left to right, which is what we need for our red pseudo-lines. Clearly, the red pseudo-lines can be drawn such as to ensure the proper intersections along each of the blue lines. It is not hard to see that this thus yields a well-separated rainbow arrangement.
\end{proof}

Next, we prove that applying the construction to a YES-instance of \AS yields a YES-instance of bicolored stretchability. The construction is similar to the proof of \Cref{lemma:bic_stretch_1}. The difference is that, by our assumption that we are given a YES-instance of \AS, we can now ensure that all red lines are straight as well.

\begin{lemma}
\label{lemma:bic_stretch_2}
    If the permutation $\pi_1, \dots, \pi_k$ can be realized by a line arrangement (with straight lines), then the corresponding instance of bicolored stretchability can also be realized with straight lines.
\end{lemma}
\begin{proof}
    We take the line arrangement that realizes $\pi_1, \dots, \pi_k$ and color all $n$ lines red. Next, choose $k$ points $p_1, \dots, p_k$ above the arrangement with increasing $x$-coordinates, corresponding to the permutations $\pi_1, \dots, \pi_k$. In other words, by walking from $p_i$ vertically downward, we are guaranteed to observe the permutation $\pi_i$, for all $i \in [k]$. Moreover, we ensure that we will not observe any crossing of two red lines while walking vertically downward. This means that we could even walk at a very small angle (i.e., almost vertically) and still observe the correct permutation.
    
    Next, we choose two points $q, q'$ below the arrangement such that $q$ is to the left of $q'$. By moving both $q$ and $q'$ vertically down towards negative infinity, we can guarantee that for all $i \in [k]$, the straight blue line $b_i$ (or $b'_i$, respectively) drawn through $p_i$ and $q$ (or $q'$, respectively) observes the permutation $\pi_i$. It remains to draw the control lines:
    \begin{itemize}
        \item $r_{-1}$ is drawn as a horizontal line just above the points $p_1, \dots, p_k$,
        \item $r_0$ is drawn as a horizontal line just below $p_1, \dots, p_k$,
        \item $r_{n + 1}$ is drawn as a horizontal line just above $q, q'$,
        \item and $r_{n + 2}$ is drawn horizontally just below $q, q'$. \qedhere
    \end{itemize}   
\end{proof}

It remains to prove the other direction, i.e.,\ that a realization of the constructed instance of bicolored stretchability also implies a realization of the \AS instance.

\begin{lemma}
\label{lemma:bic_stretch_3}
    Assume that given $\pi_1, \dots, \pi_k$, the constructed instance of bicolored stretchability can be realized using straight lines. Then there also exists a straight-line realization of the \AS instance $\pi_1, \dots, \pi_k$. 
\end{lemma}
\begin{proof}
    Consider the straight-line realization of the bicolored stretchability instance. By construction of the control lines $r_{-1}$ and $r_0$, $b_i$ and $b'_i$ have to cross at a point $p_i$ that lies above the lines $r_1, \dots, r_n$ for all $i \in [k]$. Moreover, we know that the control line $r_{n + 2}$ lies below the other lines and intersects the blue lines in the order $b_k, \dots, b_1, b'_k, \dots, b'_1$. In particular, let $q$ be an arbitrary point on $r_{n + 2}$ that lies in-between $b_1$ and $b'_k$. 

    Next, for each $i \in [k]$, draw a new line $b''_i$ through $p_i$ and $q$. Observe that $b''_i$ intersects $r_{n + 2}$ in $q$, and in particular $q$ lies between the intersections of $r_{n + 2}$ with $b'_i$ and $b_i$, respectively. Since $b_i$ and $b'_i$ intersect $r_1, \dots, r_n$ in the same order (given by $\pi_i$), so does $b''_i$. This means that each of the lines $b''_i$ observes the correct permutation, and they all pass through the same point $q$ that lies on $r_{n + 2}$. Thus, we can apply a projective transformation that turns $r_{n + 2}$ into the line at infinity, which in turn ensures that the lines $b''_1, \dots, b''_k$ are parallel. We conclude that the red lines $r_1, \dots, r_n$ after the transformation are a realization of the original \AS instance.   
\end{proof}

Putting all of this together, we conclude that bicolored stretchability is indeed $\ER$-complete. 

\begin{theorem}
\label{theorem:er-hardness}
    Bicolored stretchability is $\ER$-complete.
\end{theorem}
\begin{proof}
    We argued that it is contained in $\ER$ by using the verification-approach in~\cite{ericksonSmoothingGapNP2024}. $\ER$-hardness follows from our reduction from allowable sequences above, where correctness follows from combining \Cref{lemma:bic_stretch_2} and \Cref{lemma:bic_stretch_3}.
\end{proof}

Certainly, \Cref{theorem:er-hardness} also implies $\ER$-hardness for the realizability problem of more general two-colored rainbow arrangements (that are not necessarily well-separated). Concretely, this also implies that realizability of bicolored order types (that were studied in, e.g.,~\cite{aichholzerBicoloredOrderTypes2024}) is $\ER$-hard. However, the well-separation condition is useful for us, as it allows us to also conclude that realizability of grid USO\footnote{A grid USO is called realizable if it is induced by an instance of the \emph{P-Matrix Generalized Linear Complementarity Problem (PGLCP)}. We do not elaborate on the details of realizability of grid USOs here and instead rely on prior work that establishes the appropriate connection with hyperplane arrangements~\cite{borzechowskiTwoChoicesAre2024}.} is $\ER$-complete.

\begin{corollary}
\label{corollary:grid_uso_hard}
    Deciding whether a given grid USO is realizable is $\ER$-complete, even in two dimensions.
\end{corollary}
\begin{proof}
    Note that containment in $\ER$ follows again by giving a real-RAM verifier, and we will not go into the details of this. Instead, we explain why USO realizability is $\ER$-hard based on prior work: It was proven in~\cite{borzechowskiTwoChoicesAre2024} that realizable grid USOs correspond to well-separated hyperplane arrangements. This allows us to reduce bicolored stretchability to realizability of two-dimensional grid USO: Given the well-separated rainbow arrangement, construct the orientation as described in \Cref{sec:uso_proof}. This must yield a two-dimensional grid USO. If this grid USO is realizable, then by~\cite{borzechowskiTwoChoicesAre2024} there exists a well-separated line arrangement corresponding to this grid USO, which hence must have the same combinatorial description as the initial well-separated rainbow arrangement. Conversely, any well-separated line arrangement with the same combinatorial information implies that the constructed orientation is realizable. 
\end{proof}



\bibliography{references}

\appendix

\section{A Gap in the Original Proof}
\label{app:gap}
In this section, we discuss a gap in the proof of \cref{thm:alphaHS} by Steiger and Zhao~\cite{steigerGeneralizedHamSandwichCuts2010}.
We start with a few notations used in the paper.
Denote by $\conv(P_1)$ the convex hull of $P_1$.
For a hyperplane $h$, $h^+$ and $h^-$ denote the half-spaces on the positive and negative sides of $h$, respectively.
For every $x \in \conv(P_1)$, 
a hyperplane $h_x$ containing $x$ is an $(a_2, \dots, a_d)$ \emph{semi-cut}, if for every $i > 1$, it contains a point of $P_i$ and $|h^+_x \cap P_i| = a_i$.

We now recall briefly the strategy of the proof of \cref{thm:alphaHS} in \cite{steigerGeneralizedHamSandwichCuts2010}.
The proof is by induction on $d$.
For the inductive step, the paper presents the following lemma (called Lemma~1 in the paper).

\begin{longlemma}
\label{lem:app_lem1}
Given $x \in \conv(P_1)$ and $(a_2, \dots , a_d)$, if there is an $(a_2, \dots , a_d)$ semi-cut
$h_x$, then it is unique.
\end{longlemma}

After that, the proof continues with the following paragraph~\cite[Page 539]{steigerGeneralizedHamSandwichCuts2010}, which we quote verbatim (except for changing the label of Lemma 1 to \cref{lem:app_lem1}).

\begin{quote}
    \small
    To advance the induction, fix $(a_1, . . . , a_d)$ and suppose that $h_x$ is a cut with these
values, $x \in P_1$. 
    By \cref{lem:app_lem1}, it is the unique semi-cut containing $x$, so suppose that there is an $(a_2, \dots , a_d)$ semi-cut $h_y$ through $y \in P_1$, $y \notin h_x$. Hyperplanes $h_x$ and $h_y$ cannot meet in $\conv(P_1)$ since any such point would be in two different $(a_2, \dots ,a_d)$ semi-cuts, violating \cref{lem:app_lem1}. But this implies that $a_1 = |P_1 \cap h^+_y|$: if $y \in h^+_x$, so is every $z \in P_1 \cap h^+_x$; if $y \in h^-_x$, so is every $z \in P_1 \cap h^-_x$. Therefore $h_x$ is unique $[\dots]$.
\end{quote}

The core purpose of this paragraph is to show that the $(a_1, \dots, a_d)$-cut must be unique:
If there exist two distinct points $x, y \in P_1$ such that the corresponding $(a_2, \dots , a_d)$ semi-cuts $h_x$ and $h_y$ are $(a_1, \dots, a_d)$-cuts, then we can arrive at a contradiction, regardless of whether $h_x$ and $h_y$ intersect in $\conv(P_1)$ or not.

We have two issues with the second last sentence of the paragraph.
Firstly, we believe that there is a typo: The point $z$ should be defined as a point in $P_1 \cap h^+_y$ in the first case and a point in $P_1 \cap h^-_y$ in the second case.
In other words, in our opinion, the intended argument here is that one of $h_x^+ \cap P_1$ and $h_y^+ \cap P_1$ must be a strict subset of the other, and hence, both $h_x$ and $h_y$ cannot have the same number points of $P_1$ above them.
Secondly, we believe this intended argument is wrong or at least incomplete.
Consider the example in \cref{fig:counterexample}.
The set $P_1$ here comprises exactly two points $x$ and $y$.
Hence, $h_x$ and $h_y$ do not intersect in $\conv(P_1)$ in this case.
Note that the two hyperplanes are oriented correctly in the figure, since both are oriented ``from blue to red, then towards the right half.''
Certainly, here, these two hyperplanes are not semi-cuts of the same values for the blue point set, which could be an additional argument to yield a contradiction.
However, we fail to see how to show this argument generally, when \cref{lem:app_lem1} does not apply to this setting.

One way to fix this proof is to show that if $h_x \cap h_y$ does not contain a point of $\conv(P_1)$ then this intersection must contain a point that \emph{could} be added to $P_1$ without destroying the well-separation assumption.
However, we are unable to find an easy proof for this.

\begin{figure}[ht!]
    \centering
    \includegraphics{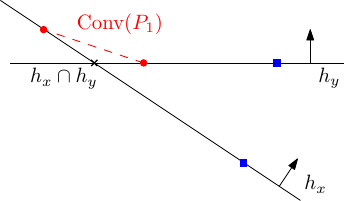}
    \caption{Example highlighting the issue in the proof of \cite{steigerGeneralizedHamSandwichCuts2010}. Red points are two disks on the left, while blue points are two squares on the right.}
    \label{fig:counterexample}
\end{figure}
\end{document}